\documentclass[11pt]{amsart}

\usepackage{fullpage, amsmath, amsthm, amssymb, color, caption, placeins}

\newcommand{\Z}{ \mathbb Z}
\newcommand{\Q}{ \mathbb Q}
\newcommand{\C}{ \mathbb C}

\newcommand{\G}{ \Gamma}

\newtheorem{theorem}{Theorem}[section]

\newtheorem{corollary}[theorem]{Corollary}

\newtheorem{lemma}[theorem]{Lemma}

\newtheorem{remark}[theorem]{Remark}

\newcommand*\HYPERskip{&}
\catcode`,\active
\newcommand*\pFq{
\begingroup
\catcode`\,\active
\def ,{\HYPERskip}%
\doHyper
}
\catcode`\,12
\def\doHyper#1#2#3#4#5{%
\, _{#1}F_{#2}\left[\begin{matrix}#3 \smallskip \\  #4\end{matrix} \; ; \; #5\right]%
\endgroup
}

\author{Holly Swisher}
\address{Department of Mathematics, Oregon State University, 368 Kidder Hall, Corvallis, OR 97331, USA}
\email{swisherh@math.oregonstate.edu}

\date{\today}



\begin{document}

\title{On the supercongruence conjectures of van Hamme}

\begin{abstract}
In 1997, van Hamme developed $p-$adic analogs, for primes $p$, of several series which relate hypergeometric series to values of the gamma function, originally studied by Ramanujan.  These analogs relate truncated sums of hypergeometric series to values of the $p-$adic gamma function, and are called Ramanujan type supercongruences.  In all, van Hamme conjectured 13 such formulas, three of which were proved by van Hamme himself, and five others have been proved recently using a wide range of methods.  Here, we explore four of the remaining five van Hamme supercongruences, revisit some of the proved ones, and provide some extensions.   
\end{abstract}

\maketitle

\section{Introduction}

In 1914, Ramanujan listed 17 infinite series representations of $1/\pi$, including for example
\[
\sum_{k= 0}^{\infty} (4k+1)(-1)^k \frac{(\frac 12)_k^3}{k!^3}= \frac 2 \pi =\frac{2}{\G\left(\frac12 \right)^2}.
\] 
Several of Ramanujan's formulas relate hypergeometric series to values of the gamma function. 

In the 1980's it was discovered that Ramanujan's formulas provided efficient means for calculating digits of $\pi$.  In 1987, J. and P. Borwein \cite{BorweinBorwein} proved all 17 of Ramanujan's identities, while D. and G. Chudnovsky \cite{ChudnovskyChudnovsky} derived additional series for $1/\pi$.  Digits of $\pi$ were calculated in both papers resulting in a new world record at the time by the Chudnovskys of $2, 260, 331, 336$ digits.  All of these Ramanujan type formulas for $1/\pi$ are related to elliptic curves with complex multiplication (CM). 

In 1997, van Hamme \cite{vanHamme} developed $p-$adic analogs, for primes $p$, of several Ramanujan type series.  Analogs of this type are called Ramanujan type supercongruences, and relate truncated sums of hypergeometric series to values of the $p-$adic gamma function.  In a recent paper \cite{CDLNS}, the author along with S. Chisholm, A. Deines, L. Long, and G. Nebe prove a general $p-$adic analog of Ramanujan type supercongruences modulo $p^2$ for suitable truncated hypergeometric series arising from CM elliptic curves.  Zudilin conjectured that the generic optimal strength in this setting should be modulo $p^3$.

In all, van Hamme conjectured 13 Ramanujan type supercongruences, which we list below in Table \ref{table}. 

We note that in the right column of Table \ref{table}, $S(m)$ denotes the corresponding sum from the left column truncated at $k=m$.  Furthermore, $a(n)$ in the last supercongruence denotes the $n$th Fourier coefficient of the eta-product 
\[
\eta(2z)^4\eta(4z)^4 = q\prod_{n\geq 1}(1-q^{2n})^4(1-q^{4n})^4= \sum_{n\geq 1} a(n)q^n,
\]
where $q=e^{2\pi iz}$.

\FloatBarrier

\scriptsize

\begin{table}
\caption{The van Hamme Conjectures} \label{table}
\begin{tabular}{llll}

 & {\bf Ramanujan Series} & &  {\bf Conjectures of van Hamme} \\ \medskip
 
(A.1) & $\sum_{k=0}^{\infty}(4k+1)(-1)^{k}\frac{\left( \frac{1}{2} \right)^5_k}{k!^5}=\frac{2}{\Gamma\big(\frac{3}{4}\big)^4}$ & (A.2) &
$S\Big(\frac{p-1}{2}\Big)\equiv\begin{cases}\frac{-p}{\Gamma_{p}\big(\frac{3}{4}\big)^4} \pmod{p^3}, &\text{if} \ \ p \equiv 1 \pmod{4}\\ 0 \ \ \ \ \ \ \ \ \ \ \pmod{p^3}, &\text{if} \ \ p \equiv 3 \pmod{4} \end{cases}$ \\ \medskip

(B.1) & $\sum_{k=0}^{\infty}(4k+1)(-1)^{k}\frac{\left( \frac{1}{2} \right)^3_k}{k!^3}=\frac{2}{\pi}=\frac{2}{\Gamma\big(\frac{1}{2}\big)^2}$ & (B.2) &
$S\Big(\frac{p-1}{2}\Big)\equiv \frac{-p}{\Gamma_{p}\big(\frac{1}{2}\big)^2} \pmod{p^3}, \ \  p\neq 2 $ \\ \medskip

(C.1) & $\sum_{k=0}^{\infty}(4k+1)\frac{\left( \frac{1}{2} \right)^4_k}{k!^4}=\infty$ & (C.2) & $S\Big(\frac{p-1}{2}\Big)\equiv p \pmod{p^3}, \ \  p\neq 2$ \\ \medskip

(D.1) & $\sum_{k=0}^{\infty}(6k+1)\frac{\left( \frac{1}{3} \right)^6_k}{k!^6}=1.01226...$ & (D.2) & $S\Big(\frac{p-1}{3}\Big)\equiv -p\Gamma_p\Big(\frac{1}{3}\Big)^9 \pmod{p^4},  \ \ \text{if} \ \ p \equiv 1 \pmod{6}$ \\ \medskip

(E.1) & $\sum_{k=0}^{\infty}(6k+1)(-1)^{k}\frac{\left( \frac{1}{3} \right)^3_k}{k!^3}=\frac{3\sqrt{3}}{2\pi}=\frac{3}{\Gamma\big(\frac{1}{3}\big)\Gamma\big(\frac{2}{3}\big)}$ & (E.2) & $S\Big(\frac{p-1}{3}\Big)\equiv p \pmod{p^3}, \ \ \text{if} \ \ p \equiv 1 \pmod{6}$ \\ \medskip

(F.1) & $\sum_{k=0}^{\infty}(8k+1)(-1)^{k}\frac{\left( \frac{1}{4} \right)^3_k}{k!^3}=\frac{2\sqrt{2}}{\pi}=\frac{4}{\Gamma\big(\frac{1}{4}\big)\Gamma\big(\frac{3}{4}\big)}$ & (F.2) & $S\Big(\frac{p-1}{4}\Big)\equiv \frac{-p}{\Gamma_p\big(\frac{1}{4}\big)\Gamma_p\big(\frac{3}{4}\big)} \pmod{p^3},  \ \ \text{if} \ \ p \equiv 1 \pmod{4}$ \\ \medskip

(G.1) & $\sum_{k=0}^{\infty}(8k+1) \frac{\left( \frac{1}{4} \right)^4_k}{k!^4}=\frac{2\sqrt{2}}{\sqrt{\pi}\Gamma\big(\frac{3}{4}\big)^2}$ & (G.2) & $S\Big(\frac{p-1}{4}\Big)\equiv p\frac{\Gamma_p\big(\frac{1}{2}\big)\Gamma_p\big(\frac{1}{4}\big)}{\Gamma_p\big(\frac{3}{4}\big)} \pmod{p^3},  \ \ \text{if} \ \ p \equiv 1 \pmod{4}$ \\ \medskip

(H.1) & $\sum_{k=0}^{\infty}\frac{\left( \frac{1}{2} \right)^3_k}{k!^3}=\frac{\pi}{\Gamma\big(\frac{3}{4}\big)^4}$ & (H.2) & $S\Big(\frac{p-1}{2}\Big)\equiv\begin{cases}-\Gamma_{p}\big(\frac{1}{4}\big)^4 \pmod{p^2}, &\text{if} \ \ p \equiv 1 \pmod{ 4}\\ 0 \ \ \ \ \ \ \ \ \ \ \pmod{p^2}, &\text{if} \ \ p \equiv 3 \pmod{ 4} \end{cases}$ \\ \medskip

(I.1) & $\sum_{k=0}^{\infty}\frac{1}{k+1}\frac{\left( \frac{1}{2} \right)^2_k}{k!^2}=\frac{4}{\pi}=\frac{4}{\Gamma\big(\frac{1}{2}\big)^2}$ & (I.2) & $S\Big(\frac{p-1}{2}\Big)\equiv 2p^2 \pmod{p^3}, \ \  p\neq 2$ \\ \medskip

(J.1) & $\sum_{k=0}^{\infty}\frac{6k+1}{4^k}\frac{\left( \frac{1}{2} \right)^3_k}{k!^3}=\frac{4}{\pi}=\frac{4}{\Gamma\big(\frac{1}{2}\big)^2}$ & (J.2) & $S\Big(\frac{p-1}{2}\Big)\equiv \frac{-p}{\Gamma_{p}\big(\frac{1}{2}\big)^2} \pmod{p^4}, \ \  p\neq 2, 3$ \\ \medskip

(K.1) & $\sum_{k=0}^{\infty}\frac{42k+5}{64^k} \frac{\left( \frac{1}{2} \right)^3_k}{k!^3}=\frac{16}{\pi}=\frac{16}{\Gamma\big(\frac{1}{2}\big)^2} $ & (K.2) & $S\Big(\frac{p-1}{2}\Big)\equiv \frac{-5p}{\Gamma_{p}\big(\frac{1}{2}\big)^2} \pmod{p^4}, \ \  p\neq 2$ \\ \medskip

(L.1) & $\sum_{k=0}^{\infty}\frac{6k+1}{8^k}(-1)^{k}\frac{\left( \frac{1}{2} \right)^3_k}{k!^3}=\frac{2\sqrt{2}}{\pi} =\frac{4}{\Gamma\big(\frac{1}{4}\big)\Gamma\big(\frac{3}{4}\big)}$ & (L.2) & $S\Big(\frac{p-1}{2}\Big)\equiv \frac{-p}{\Gamma_{p}\big(\frac{1}{4}\big)\Gamma_{p}\big(\frac{3}{4}\big)} \pmod{ p^3}, \ \  p\neq 2$ \\  \medskip

(M.1) & $\sum_{k=0}^{\infty}\frac{\left( \frac{1}{2} \right)^4_k}{k!^4}: \text{unknown}$ & (M.2) & $S\Big(\frac{p-1}{2}\Big)\equiv a(p) \pmod{p^3}, \  \  p\neq 2$ \\ \medskip
\end{tabular}
\end{table}

\FloatBarrier

\normalsize

Proofs of the supercongruences labeled (C.2), (H.2), and (I.2) were given by van Hamme.  Kilbourn \cite{Kilbourn} proved (M.2) via a connection to Calabi-Yau threefolds over finite fields, making use of the fact that the Calabi-Yau threefold in question is modular, which was proved by Ahlgren and Ono \cite{AhlgrenOno}, van Geemen and Nygaard \cite{vanGeemenNygaard}, and Verrill \cite{Verrill}.  The conjectures, (A.2), (B.2) and (J.2), have been proved using a variety of techniques involving hypergeometric series.  McCarthy and Osburn \cite{McCarthyOsburn} proved (A.2) using Gaussian hypergeometric series.  The supercongruence (B.2) has been proved in three ways, by Mortenson \cite{Mortenson} using a technical evaluation of a quotient of Gamma functions, by Zudilin \cite{Zudilin} using the W-Z method, and by Long \cite{Long} using hypergeometric series identities and evaluations.  Long also uses a similar but more general method in \cite{Long} to prove (J.2).  Furthermore, (D.2) has now been proved by Long and Ramakrishna in a recent preprint \cite{LongRamakrishna}.  In addition, they prove that (H.2) holds modulo $p^3$ when $p\equiv 1 \pmod4$, and provide extensions for (D.2) and (H.2) to additional primes.

This leaves five congruences left to prove: (E.2), (F.2), (G.2), (K.2), and (L.2).  In this paper, we first observe that Long's method to prove (B.2) in \cite{Long} can be used to prove (E.2), (F.2), (G.2), and (L.2) as well.  Furthermore we extend (E.2), (F.2), and (G.2) to results for additional primes, and show that (G.2) holds in fact modulo $p^4$.  We also revisit (A.2) to show it holds in fact modulo $p^5$ when $p \equiv 1\pmod{4}$.  In particular, we prove the following theorems.

\begin{theorem}\label{EF}
Let $a\in\{\frac12, \frac13, \frac14\}$, and $p$ an odd prime (we require $p\geq 5$ when $a=1/4$).  Let $b=1$ when $p\equiv 1 \pmod{\frac{1}{a}}$, and let $b=\frac{1}{a} -1$ when $p\equiv -1\pmod{\frac1a}$.  Then
\[
\sum_{k=0}^{a(bp-1)} \left(\frac{2k}a +1\right)(-1)^k\frac{(a)_k^3}{k!^3}  \equiv (-1)^{a(bp-1)} p \cdot b = \frac{-pb}{\G_p(a)\G_p(1-a)} \pmod{p^3}.
\]
\end{theorem}

We observe that when $a=\frac12, \frac13, \frac14$ and $p\equiv 1 \pmod{\frac1a}$, Theorem \ref{EF} gives (B.2), (E.2), and (F.2), respectively.  Furthermore, for primes $p \equiv 2 \pmod{3}$, Theorem \ref{EF} yields the following new generalization of (E.2)
\begin{equation}\label{Egen}
\sum_{k=0}^{\frac{2p-1}3} \left(6k+1\right)(-1)^k\frac{(\frac13)_k^3}{k!^3} \equiv -2p \pmod{p^3}.
\end{equation}
Similarly, for primes $p\equiv 3 \pmod{4}$, Theorem \ref{EF} yields the following new generalization of (F.2)
\begin{equation}\label{Fgen}
\sum_{k=0}^{\frac{3p-1}4} \left(8k+1\right)(-1)^k\frac{(\frac14)_k^3}{k!^3} \equiv \frac{-3p}{\G_p(\frac14)\G_p(\frac34)} = 3 \left ( \frac{-2}p \right ) p \pmod{p^3}.
\end{equation}

\begin{theorem}\label{CG}
Let $a\in\{\frac12, \frac14\}$, and $p$ an odd prime (we require $p\geq 5$ when $a=1/4$).  Let $b=1$ when $p\equiv 1 \pmod{\frac{1}{a}}$, and let $b=\frac{1}{a} -1$ when $p\equiv -1\pmod{\frac1a}$.  Then
\[
\sum_{k=0}^{a(bp-1)} \left(\frac{2k}a+1\right)\frac{(a)_k^4}{k!^4} \equiv -(-1)^{a(bp-1)} p\cdot b\delta \cdot \G_p(1-2a) \G_p(a)^2 \pmod{p^4},
\]
where $\delta = \delta_{ab}=1$ when $(a,b)\in \{(\frac12,1), (\frac14,1)\}$, and $\delta =p/2$ when $(a,b)=(\frac14,3)$.
\end{theorem}

We observe that when $a=\frac12$, Theorem \ref{CG} gives (C.2) modulo the stronger power $p^4$.  When $a = \frac14$ and $p\equiv 1 \pmod{4}$,  Theorem \ref{CG} gives (G.2) modulo the stronger power $p^4$.  Moreover, when $a = \frac14$ and $p\equiv -1 \pmod{4}$ for $p\geq 5$,  Theorem \ref{CG} gives the following new generalization of (G.2) modulo $p^4$,
\begin{equation}\label{G2Gen}
\sum_{k=0}^{\frac{3p-1}{4}} (8k+1)\frac{(\frac{1}{4})_k^4}{k!^4} \equiv -\frac{3}{2}p^2 \cdot (-1)^{\frac{3p-1}{4}} \cdot \G_p\left(\frac{1}{2}\right)\G_p\left(\frac{1}{4}\right)^2 \pmod{p^3}.
\end{equation}

\begin{theorem}\label{L}
For any odd prime $p$,
\[
\sum_{k=0}^{\frac{p-1}2}(6k+1)\left(\frac{-1}{8}\right)^k\frac{(\frac12)_k^3}{k!^3} \equiv \frac{-p}{\Gamma_{p}\left(\frac{1}{4}\right)\Gamma_{p}\left(\frac{3}{4}\right)}  = \left ( \frac{-2}p \right ) p \pmod{p^3}.
\]
\end{theorem}

Theorem \ref{L} yields the following corollary, which is difficult to prove otherwise (see Remark 1 of \cite{Long}).

\begin{corollary}\label{Lcor}
For any odd prime $p$,
\[
\sum_{k=0}^{\frac{p-1}2} (6k+1)\frac{(\frac12)_k^3}{(k!)^3}\left [ \sum_{j=1}^k \left (\frac1{(2j-1)^2}-\frac1{16j^2}\right ) \right ]\left(-\frac18\right)^k\equiv 0 \pmod p.
\]
\end{corollary}

We also have the following theorem which strengthens the (A.2) congruence when $p \equiv 1\pmod{4}$ to a congruence modulo $p^5$.

\begin{theorem}\label{A}
For any prime $p\equiv 1 \pmod{4}$, with $p>5$ 
\[
\sum_{k=0}^{\frac{p-1}{2}} (4k+1)(-1)^k\frac{(\frac{1}{2})_k^5}{k!^5} \equiv -p\cdot \G_p\left(\frac14\right)^4 \pmod{p^5}.
\]
\end{theorem}

In Section \ref{preliminaries} we discuss the gamma and $p$-adic gamma functions, as well as some useful lemmas.  In Sections \ref{E2F2}-\ref{A2} we prove our results.  In Section \ref{conjectures}, we conclude with some more general conjectures, which are supported by computational evidence from work done in Sage. 

This leaves only (K.2) from the original van Hamme conjectures, which doesn't seem to yield to this method.  As mentioned earlier, since this case corresponds to a CM elliptic curve, we know that (K.2) holds modulo $p^2$ by \cite{CDLNS}; however it remains to be proved modulo $p^4$ as conjectured.  

\section{Preliminaries}\label{preliminaries}

In this section we review hypergeometric series notation, some facts about the gamma function $\G(z)$, and the $p-$adic gamma function $\G_p(z)$.  First, we give a lemma that will be important for us later.  Recall the definition of the rising factorial for a positive integer $k$,
\[
(a)_k := (a)(a+1)\cdots(a+k-1).
\]
For $r$ a nonnegative integer and $\alpha_i,\beta_i \in \mathbb{C}$, the hypergeometric series $_{r+1}F_{r}$ is defined by
\[
\pFq{r+1}{r}{\alpha_1, \ldots , \alpha_{r+1}}{ \beta_1 , \ldots , \beta_r}{\lambda} =
\sum_{k= 0}^{\infty} \frac{(\alpha_1)_k(\alpha_2)_k \ldots (\alpha_{r+1})_k}{(\beta_1)_k \ldots (\beta_r)_k} \cdot \frac{\lambda^k}{k!},
\] which converges for $|\lambda|<1$. We write $$\pFq{r+1}{r}{\alpha_1, \ldots , \alpha_{r+1}}{ \beta_1 , \ldots , \beta_r}{\lambda}_n = \sum_{k= 0}^{n} \frac{(\alpha_1)_k(\alpha_2)_k \ldots (\alpha_{r+1})_k}{(\beta_1)_k \ldots (\beta_r)_k} \cdot \frac{\lambda^k}{k!},
$$ to denote the truncation of the series after the $\lambda^n$ term.

\begin{lemma}\label{cancelation}
Let $p$ be prime and $\zeta$ a primitive $n$th root of unity for some positive integer $n$.  If $a,b \in\Q \cap \Z_p^\times$ and $k$ is a positive integer such that $(a+j)\in\Z_p^\times$ for each $0\leq j\leq k-1$, then
\[
(a-b p)_k (a-b\zeta p)_k \cdots (a-b\zeta^{n-1} p)_k \equiv (a)_k^n  \pmod{p^n},
\]
and does not vanish modulo $p$.
Moreover for an indeterminate $x$, 
\[
(a-b x)_k (a-b\zeta x)_k \cdots (a-b\zeta^{n-1} x)_k \in \Z_p[[x^n]],
\] 
and is invertible in $\Z_p[[x^n]]$.
\end{lemma}

\begin{proof}
Expanding each term as a rising factorial, we can write
\begin{equation}\label{symmetric}
(a-b p)_k (a-b\zeta p)_k \cdots (a-b\zeta^{n-1} p)_k = \prod_{j=0}^{k-1} \prod_{i=0}^{n-1}((a+j)-b\zeta^ip).
\end{equation}
Let $\sigma_i(x_0, \ldots, x_{n-1})$ denote the $i$th elementary symmetric polynomial in $n$ variables.  Then we have $\sigma_i(1, \zeta, \ldots, \zeta^{n-1})=0$, for $1\leq i\leq n-1$ and $\sigma_n(1, \zeta, \ldots, \zeta^{n-1}) = \pm 1$.  For a fixed $0\leq j\leq k-1$, we thus have
\[
\prod_{i=0}^{n-1}((a+j)-b\zeta^ip) = (a+j)^n \pm b^np^n \equiv (a+j)^n \pmod{p^n}.
\]
Together with \eqref{symmetric} we see that the result holds, and is nontrivial precisely when $(a+j)\in\Z_p^\times$ for each $0\leq j\leq k-1$.  Replacing $p$ by $x$ in \eqref{symmetric} gives the additional result for series.  Since the constant term for $(a-b x)_k (a-b\zeta x)_k \cdots (a-b\zeta^{n-1} x)_k$ is $\prod_{j=0}^{k-1} (a+j)^n \in\Z_p^\times$, the series is invertible in $\Z_p[[x^n]]$.
\end{proof}

\subsection{The gamma function}

The gamma function $\G(z)$ is a meromorphic function on $\C$ with (simple) poles precisely at the nonpositive integers, which extends the factorial function on positive integers, namely 
\begin{equation}\label{factorial}
\G(n) = (n-1)!
\end{equation}  
for positive integers $n$.  It also satisfies the functional equation 
\begin{equation}\label{functional}
\G(z+1) = z\G(z),
\end{equation}
which immediately yields that for complex $z$ and positive integers $k$,

\begin{equation}\label{Gquotient}
\frac{\G(z+k)}{\G(z)} = (z)_k.
\end{equation}

\noindent Also important is the following reflection formula due to Euler.  For complex $z$,
\begin{equation}\label{Greflection}
\G(z)\G(1-z)=\frac{\pi}{\sin(\pi z)}.
\end{equation}
Furthermore, $\G$ satisfies the duplication formula
\begin{equation}\label{Gduplication}
\G(z)\G\left(z+\frac12\right)=2^{1-2z}\sqrt{\pi} \G(2z).
\end{equation}

\noindent The following lemma will be useful in the next section.

\begin{lemma}\label{Glemma}
Let $a,b\in\Q$ and $p$ prime such that $a(bp-1)$ is a positive integer.  Then
\[
\frac{\G(1-abp)\G(1+abp)}{\G(1-a)\G(1+a)} = (-1)^{a(bp-1)}\cdot bp
\]
\end{lemma}
\begin{proof}
By (\ref{Gquotient}), we have that
\[
\frac{\G(1+abp)}{\G(1+a)} = (1+a)_{a(bp-1)}
\]
and
\[
\frac{\G(1-a)}{\G(1-abp)} = (1-abp)_{a(bp-1)} = (-1)^{a(bp-1)} (a)_{a(bp-1)},
\]
which gives the desired result.
\end{proof}

\subsection{The $p-$adic gamma function}
We note that many of the facts we state in this section can be found in Morita \cite{Morita}.  Let $p$ be an odd prime.  Set $\G_p(0)=1$, and for positive integers $n$, define
\begin{equation}\label{posGp}
\G _p(n) = (-1)^{n}\prod_{\substack{0<  j < n \\ p \, \nmid j}}j.
\end{equation}
The $p-$adic gamma function is the extension to $\Z_p$ defined by 
\[
\G _p(\alpha) = \lim_{n\rightarrow\alpha}\G_p(n),
\]
where $n$ are positive integers $p-$adically approaching $\alpha$.  With this definition, $\G_p(\alpha)$ is a uniquely defined continuous function on $\Z_p$.

For $\alpha\in\Z_p$, the following fact is found in \cite{Dwork}:
\begin{equation}\label{Gpfunctional}
\frac{\G_p(\alpha + 1)}{\G_p(\alpha)} = \begin{cases}
    -\alpha  & \text{ if } \alpha \not\in p\Z_p,\\
    -1  & \text{if } \alpha \in p\Z_p.
\end{cases}
\end{equation}
From (\ref{Gpfunctional}), the following lemma follows immediately.
\begin{lemma}\label{Gpquotient}
For a positive integer $k$, and $\alpha\in \Z_p$, if $\alpha, \alpha+1,\ldots, \alpha + k - 1 \not\in p\Z_p$, then
\[
\frac{\G_p(\alpha+k)}{\G_p(\alpha)} = (-1)^k(\alpha)_k.
\]
More generally, 
\[
\frac{\G_p(\alpha+k)}{\G_p(\alpha)} =(-1)^k \prod_{\substack{j=0 \\ \alpha+j\not\in p\Z_p}}^{k-1}(\alpha+j).
\]
\end{lemma}

\noindent We also have for $x\in\Z_p$ that 
\begin{equation}\label{Gpreflection}
\G_p(x)\G_p(1-x) = (-1)^{a_0(x)},
\end{equation}
where $1\leq a_0(x) \leq p$ is the least positive residue of $x$ modulo $p$.

Let $G_k(a)= \G_p^{(k)}(a)/\G_p(a)$, where $\G_p^{(k)}$ denotes the $k$th derivative of $\G_p$, and $G_0(a)=1$.  Long and Ramakrishna \cite{LongRamakrishna} show that for $a \in\Z_p$,
\begin{equation}\label{Gfacts}
G_0(a) =1, \quad G_1(a) = G_1(1-a), \quad G_2(a) + G_2(1-a) = 2G_1(a)^2.
\end{equation}
Furthermore, they prove the following useful theorem.
\begin{theorem}\label{Gpsum}(Long, Ramakrishna \cite{LongRamakrishna})
Let $p\geq 5$ be prime, $r$ a positive integer, $a,b\in \Q \cap \Z_p$, and $t\in\{0,1,2\}$.  Then
\[
\frac{\G_p(a+bp^r)}{\G_p(a)} \equiv \sum_{k=0}^t \frac{G_k(a)}{k!}(bp^r)^k\pmod{p^{(t+1)r}}.
\]
\end{theorem}

\begin{remark}\label{LRext}
Fixing $r=1$, we can extend this result modulo $p^4$ when $p>5$.  Let $a\in\Q\cap\Z_p$, $b\in\Z_p$, and let $v_p(x)$ denote the $p$-order of $x$.  From the proof of Theorem \ref{Gpsum} in \cite{LongRamakrishna} (using also their Proposition 14) we have that
\[
\frac{\G_p(a + bp)}{\G_p(a)} = \sum_{k=0}^\infty \frac{G_k(a)}{k!} (bp)^k.
\] 
We can show that $v_p(\frac{G_k(a)}{k!} (bp)^k) \geq 4$ when $p>5$.  This is because as in \cite{LongRamakrishna}, $v_p(\frac{G_k(a)}{k!})=0$ for all $k<p$, and in general, $v_p(\frac{G_k(a)}{k!})\geq -k(\frac1p + \frac1{p-1})$.  So we have that $v_p(\frac{G_k(a)}{k!} (bp)^k) \geq 4$ when $4\leq k<p$.  For $k\geq p$, we have that
\[
v_p\left(\frac{G_k(a)}{k!} (bp)^k\right) \geq k - k\left( \frac1p + \frac1{p-1} \right),
\]
and so the inequality we need is
\begin{equation}\label{Aineq}
1 \geq \left(\frac{k}{k-4} \right) \left( \frac1p + \frac1{p-1} \right).
\end{equation}
The inequality \eqref{Aineq} holds for all $k\geq 5$ when $p\geq 11$, and holds for all $k\geq 6$ when $p=7$.  When $p=5$, we see that \eqref{Aineq} holds for all $k\geq 8$, which leaves the cases $k=5,6,7$.  When $k=6,7$ a calculation shows that $v_p(\frac{G_k(a)}{k!} (bp)^k) \geq 4$, however the $k=5$ case remains elusive.  We thus obtain that for primes $p>5$,
\begin{equation}\label{exteq}
\frac{\G_p(a+bp)}{\G_p(a)} \equiv \sum_{k=0}^3 \frac{G_k(a)}{k!}(bp)^k\pmod{p^4}.
\end{equation} 
\end{remark}

%CHECK!
%Thus we observe that due to cancelation,
%\begin{equation}
%\G_p(a+bp^r)\G_p(a-bp^r) \equiv \G_p(a)^2 \pmod{p^{2r}}.
%\end{equation}

%%%%%%%
\section{Proof of (B.2), (E.2) and (F.2) with generalizations}\label{E2F2}  

\noindent For this section we will make use of the following identities of Whipple (see (5.1) and (6.3) in \cite{Whipple})

\begin{equation}\label{whipple}
\pFq{4}{3}{\frac{a}{2}+1, a, c, d}{,\frac{a}{2}, 1+a-c, 1+a-d}{-1} = \frac{\G(1+a-c)\G(1+a-d)}{\G(1+a)\G(1+a-c-d)},
\end{equation} 
and 
\begin{multline}\label{6F5identity}
\pFq{6}{5}{\frac{a}{2}+1& a& c& d& e& f}{&\frac{a}{2}& 1+a - c& 1+a - d& 1+a - e& 1+a - f}{-1} \\ = \frac{\G(1+a-e)\G(1+a-f)}{\G(1+a)\G(1+a-e-f)}\cdot \pFq{3}{2}{1+a-c-d& e& f}{& 1+a-c& 1+a-d}{1}.
\end{multline}

Observe that the left hand side targets for (B.2), (E.2), and (F.2) can all be expressed by the truncated hypergeometric series
\begin{equation}\label{EFLHS} 
\sum_{k=0}^{a(p-1)} \left(\frac{2k}a+1\right)(-1)^k\frac{(a)_k^3}{k!^3} = \pFq{4}{3}{\frac{a}{2}+1, a, a, a}{,\frac{a}{2}, 1, 1}{-1}_{a(p-1)},
\end{equation}
where $a\in\{\frac12, \frac13, \frac14\}$, and $p$ is a prime such that $p\equiv 1\pmod{\frac1a}$.  We will also consider primes $p$ for which $p\equiv -1\pmod{\frac1a}$ to obtain additional van Hamme type supercongruences related to (E.2) and (F.2).  We now prove Theorem \ref{EF}

\begin{proof}[Proof of Theorem \ref{EF}]
For fixed $a\in\{\frac12, \frac13, \frac14\}$ and an odd prime $p$, let $b\in\Q$ be defined by 
\[
b=b(a,p)=\begin{cases}
1      & \text{if }  p\equiv 1\pmod{\frac1a}, \\
 \frac{1}{a} -1 & \text{if } p\equiv -1\pmod{\frac1a},
\end{cases}
\]
so that $a(1-bp)$ is a negative integer, and $b\in\Z_p^\times$ (we require $p\geq 5$ when $a=1/4$).  Consider the hypergeometric series
\[
F(p) = \pFq{4}{3}{\frac{a}{2}+1, a, a(1-bp), a(1+bp)}{,\frac{a}{2}, 1+abp, 1-abp}{-1}.
\]
Then $F(p)$ naturally truncates at $a(bp-1)$ since $a(1-bp)$ is a negative integer.

Letting $c=a(1-bp), d=a(1+bp)$ in (\ref{whipple}), we get using Lemma \ref{Glemma}
\begin{equation}\label{EFequality}
F(p) = \frac{\G(1+abp)\G(1-abp)}{\G(1+a)\G(1-a)}  = (-1)^{a(bp-1)}\cdot bp.
\end{equation}
When $p\equiv 1 \pmod{\frac{1}{a}}$, \eqref{Gpreflection} shows that $F(p)$ is the right hand side target for (B.2), (E.2), and (F.2).  

Switching $p$ to a variable $x$ and truncating at $k=a(bp-1)$, we define 
\[
F(x) = \pFq{4}{3}{\frac{a}{2}+1, a, a(1-bx), a(1+bx)}{,\frac{a}{2}, 1+abx, 1-abx}{-1}_{a(bp-1)}.
\]
By Lemma \ref{cancelation}, $F(x)\in\Z_p[[x^2]]$, and so $F(x)= C_0 + C_2x^2 + C_4x^4 + \cdots$, for $C_i\in\Z_p$.  Notice that $C_0$ is precisely our left hand side target \eqref{EFLHS}.  Thus if $p\mid C_2$, then letting $x=p$ gives the desired congruence $F(p) \equiv C_0 \pmod{p^3}$.

To see that $p\mid C_2$, let $c=a(1-bx)$, $d=a(1+bx)$, $e=1$, and $f=a(1-bp)$ in (\ref{6F5identity}).  We then have by (\ref{functional}),
\begin{multline}\label{6F5}
\pFq{6}{5}{\frac{a}{2}+1& a&  a(1-bx) & a(1+bx) & 1 & a(1-bp)}{&\frac{a}{2}& 1+abx& 1-abx & a & 1+abp}{-1} \\ = \frac{\G(a)\G(1+abp)}{\G(1+a)\G(abp)}\cdot \pFq{3}{2}{1-a & 1& a(1-bp)}{& 1+abx& 1-abx}{1} \\ = bp\cdot \pFq{3}{2}{1-a & 1& a(1-bp)}{& 1+abx& 1-abx}{1}.
\end{multline}
The hypergeometric series on both sides of \eqref{6F5} naturally truncate at $a(bp-1)$ since $a(1-bp)$ is a negative integer.  Also, modulo $p$, the left hand side is congruent to $F(x)$.  Thus, as a series in $x^2$, we have
\[
F(x) \equiv bp\cdot \pFq{3}{2}{1-a & 1& a(1-bp)}{& 1+abx& 1-abx}{1} \pmod{p}.
\]
Thus $p\mid C_2$ and we have proven Theorem \ref{EF}.
\end{proof}

%%%%%%%
\section{Proof of (C.2) and (G.2) modulo $p^4$ with a generalization}\label{G2}

\noindent For this section we use the following identities of Whipple (see (5.2) and (7.7) in \cite{Whipple}) 

\begin{multline}\label{5F4identity} \pFq{5}{4}{\frac{a}{2}+1& a& c& d& e}{&\frac{a}{2}&1+a - c& 1+a - d& 1+a - e}{1} \\  = \frac{\G(1+a-c)\G(1+a-d)\G(1+a-e)\G(1+a-c-d-e)}{\G(1+a)\G(1+a-d-e)\G(1+a-c-d)\G(1+a-c-e)},
\end{multline}
and
\begin{multline}\label{7F6identity} 
\pFq{7}{6}{\frac{a}{2}+1& a& c& d& e & f & g}{&\frac{a}{2}&1+a - c& 1+a - d& 1+a - e & 1+a - f& 1+a - g}{1} \\  = \frac{\G(1+a-e)\G(1+a-f)\G(1+a-g)\G(1+a-e-f-g)}{\G(1+a)\G(1+a-f-g)\G(1+a-e-f)\G(1+a-e-g)} \\  \cdot \pFq{4}{3}{1+a-c-d& e& f & g}{&e+f+g-a &1+a-c& 1+a-d}{1}, 
\end{multline}
provided the $_4F_3$ series terminates.

As in Section \ref{E2F2}, we first observe that the left hand side targets for (C.2), and (G.2) can be expressed by the truncated hypergeometric series
\begin{equation}\label{CGLHS} 
\sum_{k=0}^{a(p-1)} \left(\frac{2k}a+1\right)\frac{(a)_k^4}{k!^4} = \pFq{5}{4}{\frac{a}{2}+1, a, a, a, a}{,\frac{a}{2}, 1, 1, 1}{1}_{a(p-1)},
\end{equation}
where $a\in\{\frac12, \frac14\}$, and $p$ is a prime such that $p\equiv 1\pmod{\frac1a}$.  We will also consider primes $p$ for which $p\equiv -1\pmod{\frac1a}$ when $a=1/4$ to obtain an additional van Hamme type supercongruence related to (G.2).  We now prove Theorem \ref{CG}.

\begin{proof}[Proof of Theorem \ref{CG}]
For fixed $a\in\{\frac12, \frac14\}$ and an odd prime $p$, define $b\in\Q$ by 
\[
b=b(a,p)=\begin{cases}
1      & \text{if }  p\equiv 1\pmod{\frac1a}, \\
 \frac{1}{a} -1 & \text{if } p\equiv -1\pmod{\frac1a},
\end{cases}
\]
so that $a(1-bp)$ is a negative integer, and $b\in\Z_p^\times$ (requiring that $p\geq 5$ when $a=1/4$).  

Let $\omega$ be a primitive third root of unity and consider the hypergeometric series 
\begin{equation}\label{G5F4(p)}
G(p) = \pFq{5}{4}{\frac{a}{2}+1, a, a(1-bp), a(1-b\omega p), a(1-b\omega^2 p)}{,\frac{a}{2}, 1+abp, 1+ab\omega p, 1+ab\omega^2 p}{1}, 
\end{equation}
which naturally truncates at $a(bp-1)$.  Moreover by Lemma \ref{cancelation}, 
\begin{equation}\label{GLHSp^2}
G(p) \equiv\pFq{5}{4}{\frac{a}{2}+1, a, a, a, a}{,\frac{a}{2}, 1, 1, 1}{1}_{a(bp-1)}  \pmod{p^3}.
\end{equation}

Letting $c=a(1-bp)$, $d=a(1-b\omega p)$, and $e=a(1-b\omega^2 p)$ in (\ref{5F4identity}), gives that
\begin{equation}\label{Gammas}
G(p) = \frac{\G(1+ abp) \G(1+ ab \omega p) \G(1+ ab\omega^2 p) \G(1-2a)}{\G(1+ a) \G(1 -a -abp) \G(1 -a -ab\omega^2 p) \G(1 -a -ab\omega p)}.
\end{equation}
We will show in \eqref{finalcong} that this gives the right hand side target from Theorem \ref{CG}.

In the meantime, as in Section \ref{E2F2}, we consider the series obtained from $G(p)$ by switching $p$ with an indeterminate $x$ and truncating at $a(bp-1)$,
\[
G(x) = \pFq{5}{4}{\frac{a}{2}+1, a, a(1-bx), a(1-b\omega x), a(1-b\omega^2 x)}{,\frac{a}{2}, 1+abx, 1+ab\omega x, 1+ab\omega^2 x}{1}_{a(bp-1)}.
\]
By Lemma \ref{cancelation}, $G(x)\in\Z_p[[x^3]]$, and so we have $G(x)= C_0 + C_3x^3 + C_6x^6 + \cdots$, for $C_i\in\Z_p$ where $C_0$ is our left hand side target \eqref{CGLHS}.  Thus if $p\mid C_3$, then letting $x=p$ gives the desired congruence $G(p) \equiv C_0 \pmod{p^4}$.  

Let $c=a(1-b\omega x)$, $d=a(1-b\omega^2 x)$,  $e=a(1-bx)$, $f=a(1- bp)$, and $g=1$ in (\ref{7F6identity}) to get 
\begin{multline}\label{7F6} 
\pFq{7}{6}{\frac{a}{2}+1 & a  & a(1-b\omega x) & a(1-b\omega^2 x)  & a(1-bx)& a(1-bp) & 1}
{& \frac{a}{2} & 1 + ab\omega x & 1+ab\omega^2 x & 1+ abx &1 + abp & a }{1} \\  
= \frac{\G(1+ abx)\G(1+ abp )\G(a)\G(a(bx+bp-1))}{\G(1+a)\G(abp)\G(abx)\G(1+ a(bx+bp-1))} \\
\cdot \pFq{4}{3}{1 -a -abx & a(1-bx)& a(1-bp) & 1}{& 1-a(bx+bp-1) &  1 + ab\omega x & 1+ab\omega^2 x}{1},
\end{multline}
where since $a(1-bp)$ is a negative integer, both sides of (\ref{7F6}) terminate at $a(bp-1)$.  
By (\ref{functional}), we have that
\[
 \frac{\G(1+ abx)\G(1+ abp )\G(a)\G(a(bx+bp-1))}{\G(1+a)\G(abp)\G(abx)\G(1+ a(bx+bp-1))} 
= p\cdot \frac{b^2x}{bx + (bp -1)}\in p \cdot \Z_p[[x]],
\]
since the integers $b$ and $(bp-1)$ are in $\Z_p^\times$.  

As series in $x^3$, 
\[
\pFq{7}{6}{\frac{a}{2}+1 & a  & a(1-b\omega x) & a(1-b\omega^2 x)  & a(1-bx)& a(1-bp) & 1}
{& \frac{a}{2} & 1 + ab\omega x & 1+ab\omega^2 x & 1+ abx &1 + abp & a }{1} \equiv G(x) \pmod{p},
\]
thus $p\mid C_0$, as desired. 

To finish the proof of Theorem \ref{CG}, it remains to show that $G(p)$ gives the appropriate right hand side target.

Fix $k=a(bp-1)$.  From \eqref{Gammas} and \eqref{Gquotient} we can rewrite $G(p)$ as
\[
G(p) = \frac{(1+a)_k (1 -a -abp)_k (1+abp)_k}{(1+ab\omega p)_k (1+ab\omega^2 p)_k (1+ab p)_k}.
\]
We first observe that since $k+1$ is a positive integer, Lemma \ref{cancelation}, \eqref{posGp}, and \eqref{Gpreflection} show that the denominator satisfies the congruence
\begin{equation}\label{denominator}
(1+ab\omega p)_k (1+ab\omega^2 p)_k (1+abp)_k \equiv (1)_k^3 = -(-1)^k\cdot\G_p(k+1)^3 = 1 / \G_p(-k)^3 \pmod{p^3}.
\end{equation}
To next evaluate the numerator in terms of $p-$adic gamma functions, we employ Lemma \ref{Gpquotient} for $\alpha\in\{1+a, 1+abp, 1 -a -abp \}$.  When $\alpha = 1+a$, the factors $\alpha, \alpha+1,\ldots, \alpha + k - 2$ are not in $p\Z_p$, but $\alpha+k-1 = abp$. 
%Note if $(p,a)\in\{(3, 1/2), (5, 1/4)$, then $\alpha = \alpha+k-1 = abp$.  
Thus using \eqref{Gpreflection},
\begin{equation}\label{alpha1}
(1+a)_k = (-1)^k \cdot abp \cdot \frac{\G_p( 1+ a + k)}{\G_p(1+a)} = -(-1)^k \cdot bp \cdot \frac{\G_p( 1+ abp)}{\G_p(a)} = bp \cdot \G_p(1-a)\G_p( 1+ abp).
\end{equation}
When $\alpha = 1+abp$, one sees that none of the factors 
$\alpha, \alpha+1,\ldots, \alpha + k - 1$ are in $p\Z_p$, so that
\begin{equation}\label{alpha2}
(1+ abp)_k = (-1)^k\cdot \frac{\G_p(1-a +2abp)}{\G_p(1+ abp)}.
\end{equation}
Finally, when $\alpha = 1 -a -abp$, we first note that for $a=1/2$, each factor $\alpha+j$ is an integer in the range $ -k < \alpha+j \leq -1$ so none are in $p\Z_p$.  When $a=1/4$, we see that each factor $\alpha+j$ is in $\frac{1}{2}\Z$.  When $p\equiv 1 \pmod{4}$, $b=1$ so each $\alpha+j$ satisfies $-p/2< \alpha +j \leq -1/2$ and so none are in $p\Z_p$.  Thus using \eqref{Gpreflection} we have when $(a,b)\in\{(\frac12,1),(\frac 14,1)\}$,
\[
(1-a-abp)_k = (-1)^k\cdot \frac{\G_p(1-2a)}{\G_p(1-a-abp)} = -\G_p(1-2a)\G_p(a+abp).
\]
When $p\equiv -1 \pmod{4}$, $b=3$, the integer $j=(p-3)/4 < k$ yields $\alpha+j = p/2$.  Since each $\alpha+j$ in this case satisfies $-p< \alpha +j \leq -1/2$, this is the only factor in $p\Z_p$. 
Thus in this case,
\[
(1-a-abp)_k = - \frac{p}{2} \cdot \G_p(1-2a)\G_p(a+abp).
\]
Putting this together, we have
\begin{equation}\label{alpha3}
(1-a-abp)_k = - \delta_{ab}\cdot \G_p(1-2a)\G_p(a+abp),
\end{equation}
where $\delta_{ab}$ is defined to be $1$ when $(a,b)\in\{(\frac12, 1), (\frac14,1)\}$, and $p/2$ when $(a,b)=(\frac14,3)$.

Combining \eqref{denominator}, \eqref{alpha1}, \eqref{alpha2}, \eqref{alpha3}, and using \eqref{Gpreflection}, the factor of (at least one) $p$ gives the following congruence modulo $p^4$,
\[
G(p) \equiv  -(-1)^k p\cdot b\delta_{ab} \cdot \G_p(1-2a) \G_p(1-a)\G_p(a - abp)^3 \G_p(a + apb) \G_p(1 - a +2abp) \pmod{p^4}.
\]
By Theorem \ref{Gpsum}, we have that
\begin{align*}
\G_p(a-abp) & \equiv \G_p(a) \cdot \left[1 - abG_1(a) p + \frac{a^2b^2}{2}G_2(a) p^2 \right] \pmod{p^3}\\
\G_p(a+abp) & \equiv \G_p(a) \cdot \left[1 + abG_1(a) p + \frac{a^2b^2}{2}G_2(a) p^2 \right] \pmod{p^3}\\
\G_p(1-a+2abp) & \equiv \G_p(1-a) \cdot \left[1 + 2abG_1(1-a) p + 2a^2b^2G_2(1-a) p^2 \right] \pmod{p^3}.
\end{align*}
Using \eqref{Gfacts}, we see that 
\[
\G_p(a-abp)^3 \G_p(a+abp) \G_p(1-a+2abp)  \equiv -(-1)^k \G_p(a)^3 \pmod{p^3},
\]
and so
\begin{equation}\label{finalcong}
G(p) \equiv -(-1)^k p\cdot b\delta_{ab}\cdot \G_p(1-2a) \G_p(a)^2 \pmod{p^4}
\end{equation}
as desired.
\end{proof}

%%%%%%%
\section{Proof of (L.2)}\label{L2}
For this section we use the following identity from \cite{Karlsson} (see (18)) which gives that
\begin{equation}\label{eq:Karlsson}
\pFq{4}{3}{3a,a+1,b,1-b}{, a,\frac{3a +b+1}2 ,\frac {3a-b+2}2}{-\frac18}=\frac{\G(\frac {3a+b+1}2 )\G(\frac {3a-b+2}2)}{\G(\frac{3a+1}2)\G(\frac{3a+2}2)}.
\end{equation}

\begin{proof}[Proof of Theorem \ref{L}]
Let $p$ be an odd prime.  Observe that using \eqref{Gpreflection}, the right hand side target for (L.2) can be written as
\[
\frac{-p}{\Gamma_{p}\left(\frac{1}{4}\right)\Gamma_{p}\left(\frac{3}{4}\right)}  = (-1)^{\left(\frac{p^2-1}8 \right) + \left(\frac{p-1}2 \right)}p = \left ( \frac{-2}p \right ) p.
\]
The left hand side target for (L.2) can be expressed by the truncated hypergeometric series
\begin{equation}\label{LLHS} 
\sum_{k=0}^{\frac{p-1}2}(6k+1)\left(\frac{-1}{8}\right)^k\frac{(\frac12)_k^3}{k!^3} = \pFq{4}{3}{\frac76, \frac12, \frac12, \frac12}{,\frac16, 1, 1}{-\frac18}_{\frac{p-1}2}.
\end{equation}

Consider the hypergeometric series
\[
L(p) = \pFq{4}{3}{\frac76, \frac12, \frac{1-p}2, \frac{1+p}2}{,\frac16, 1-\frac{p}4, 1+\frac{p}4}{-\frac18}.
\]
Then $L(p)$ naturally truncates at $\frac{p-1}2$ since $\frac{1-p}2$ is a negative integer.  As shown in \cite{Long} (Lemma 4.4), 
\[
L(p) = \left ( \frac{-2}p \right ) p,
\]
our right hand side target for (L.2).  Switching $p$ to a variable $x$ and truncating at $k=\frac{p-1}2$, we define 
\[
L(x) =\pFq{4}{3}{\frac76, \frac12, \frac{1-x}2, \frac{1+x}2}{,\frac16, 1-\frac{x}4, 1+\frac{x}4}{-\frac18}_{\frac{p-1}2}.
\]
By Lemma \ref{cancelation}, $L(x)\in\Z_p[[x^2]]$, and so $L(x)= C_0 + C_2x^2 + C_4x^4 + \cdots$, for $C_i\in\Z_p$.  Notice that $C_0$ is precisely our left hand side target \eqref{LLHS}.  Thus if $p\mid C_2$, then letting $x=p$ gives our desired congruence $L(p) \equiv C_0 \pmod{p^3}$.

To see that $p\mid C_2$, observe that as a series in $x^2$, 
\[
\pFq{4}{3}{\frac{7-p}6, \frac{1-p}2, \frac{1-x}2, \frac{1+x}2}{,\frac{1-p}6, 1-\frac{p}4-\frac{x}4, 1-\frac{p}4+\frac{x}4}{-\frac18} \equiv L(x) \pmod{p},
\]
where since the left hand side is naturally truncating at $\frac{p-1}2$, we actually have that it is a rational function in $x$.  Moreover, letting $a=\frac{1-p}6$ and $b=\frac{1-x}2$ in \eqref{eq:Karlsson}, we see that this rational function is actually $0$, since
\[
\pFq{4}{3}{\frac{7-p}6, \frac{1-p}2, \frac{1-x}2, \frac{1+x}2}{,\frac{1-p}6, 1-\frac{p}4-\frac{x}4, 1-\frac{p}4+\frac{x}4}{-\frac18} = \frac{\G_p(1-\frac{p}4-\frac{x}4)\G_p(1-\frac{p}4+\frac{x}4)}{\G_p(\frac{3-p}4)\G_p(\frac{5-p}4)}.
\]
and one of $\frac{3-p}4$ or $\frac{5-p}4$ is a nonpositive integer, yielding a pole for $\G$. 
\end{proof}

\begin{proof}[Proof of Corollary \ref{Lcor}]
By Lemma \ref{cancelation}, we have directly that $L(p)$ is congruent to our left hand side target modulo $p^2$.  Considering the difference,
\begin{multline*}
\sum_{k=0}^{\frac{p-1}2}(6k+1)\left(\frac{-1}{8}\right)^k\frac{(\frac12)_k^3}{k!^3} - L(p) = \sum_{k=0}^{\frac{p-1}2}(6k+1)\left(\frac{-1}{8}\right)^k\frac{(\frac12)_k}{k!}\left[\frac{(\frac12)_k^2}{(k!)^2}  - \frac{\left(\frac{1-p}2\right)_k\left(\frac{1+p}2\right)_k}{\left(1 - \frac{p}4\right)_k\left(1 + \frac{p}4\right)_k} \right] \\
= \sum_{k=0}^{\frac{p-1}2}(6k+1)\left(\frac{-1}{8}\right)^k\frac{(\frac12)_k}{k!}\left[\frac{\left(\frac12\right)_k^2 \left(1 - \frac{p}4\right)_k\left(1 + \frac{p}4\right)_k - (k!)^2\left(\frac{1-p}2\right)_k\left(\frac{1+p}2\right)_k}{(k!)^2\left(1 - \frac{p}4\right)_k\left(1 + \frac{p}4\right)_k} \right],
\end{multline*}
we observe that due to cancelation, 
\begin{multline*}
\left(\frac12\right)_k^2 \left(1 - \frac{p}4\right)_k\left(1 + \frac{p}4\right)_k - (k!)^2\left(\frac{1-p}2\right)_k\left(\frac{1+p}2\right)_k  \\
= p^2(k!)^2\left(\frac12\right)_k^2\left[ \sum_{j=1}^k \left(\frac1{(2j-1)^2}-\frac1{16j^2}\right) \right] + \cdots,
\end{multline*}
where the remaining terms all have a factor of $p^n$ for $n\geq 4$.  Also we observe that the denominator $(k!)^2\left(1 - \frac{p}4\right)_k\left(1 + \frac{p}4\right)_k$ contains no factors of $p$.  Thus as a corollary to Theorem \ref{L} we obtain the desired result, 
\begin{equation}
\sum_{k=0}^{\frac{p-1}2} (6k+1)\frac{(\frac12)_k^3}{(k!)^3}\left [ \sum_{j=1}^k \left (\frac1{(2j-1)^2}-\frac1{16j^2}\right ) \right ]\left(-\frac18\right)^k\equiv 0 \pmod p.
\end{equation}
\end{proof}

%%%%%%%
\section{Proof of (A.2) modulo $p^5$ for $p\equiv 1 \pmod{4}$}\label{A2}

For this section we use the following identity from \cite{AAR} (see Thm. 3.5.5 (ii)), which gives that 
\begin{equation}\label{AAR3.5.5}
\pFq{3}{2}{a, b, c}{, e, f}{1} = \frac{\pi \G(e) \G(f)}{2^{2c-1}\G(\frac{a+e}2) \G(\frac{a+f}2) \G(\frac{b+e}2) \G(\frac{b+f}2)},
\end{equation}
when $a+b=1$ and $e+f=2c+1$.

\begin{proof}[Proof of Theorem \ref{A}]
Let $p\equiv 1 \pmod{4}$ be prime, with $p>5$.  Observe that the left hand side target for (A.2) can be expressed by
\begin{equation}\label{ALHS} 
\sum_{k=0}^{\frac{p-1}2} \left(4k+1\right)(-1)^k\frac{(\frac12)_k^5}{k!^5} = \pFq{6}{5}{\frac54, \frac12, \frac12, \frac12, \frac12, \frac12}{,\frac14, 1, 1, 1, 1}{-1}_{\frac{p-1}2}.
\end{equation}
Consider the hypergeometric series 
\begin{equation}\label{A6F5(p)}
A(p) = \pFq{6}{5}{\frac54, \frac12, \frac{1-ip}2, \frac{1+ip}2, \frac{1-p}2, \frac{1+p}2}{,\frac14, 1+\frac{ip}2, 1-\frac{ip}2, 1+\frac{p}2, 1-\frac{p}2}{-1}, 
\end{equation}
which naturally truncates at $\frac{p-1}2$.  By Lemma \ref{cancelation}, 
\begin{equation}\label{ALHSp^4}
A(p) \equiv \pFq{6}{5}{\frac54, \frac12, \frac12, \frac12, \frac12, \frac12}{,\frac14, 1, 1, 1, 1}{-1}_{\frac{p-1}2} \pmod{p^4}.
\end{equation}

Letting $c= \frac{1-ip}2$, $d=\frac{1+ip}2$, $e=\frac{1-p}2$, and $f=\frac{1+p}2$ in (\ref{6F5identity}), gives that
\begin{equation*}
A(p) = \frac{\G\left(1 + \frac{p}2 \right)\G\left(1 - \frac{p}2 \right)}{\G\left(\frac32 \right)\G\left( \frac12\right)} \cdot \pFq{3}{2}{\frac12, \frac{1-p}2, \frac{1+p}2}{,1 + \frac{ip}2, 1 - \frac{ip}2}{1}.
\end{equation*}
Note that by \eqref{functional} and \eqref{Greflection}, $\G\left(\frac32 \right)\G\left( \frac12\right)=\frac{\pi}2$.  Thus, letting $a=\frac{1-p}2$, $b=\frac{1+p}2$, $c = \frac12$, $e = 1 + \frac{ip}2$, and $f = 1 - \frac{ip}2$ in \eqref{AAR3.5.5}, we see that $a+b=1$ and $e+f=2c$, so we obtain
\begin{equation}\label{AGammas}
A(p) = \frac{2\cdot \G\left(1 + \frac{p}2 \right) \G\left(1 - \frac{p}2  \right) \G\left(1 + \frac{ip}2  \right) \G\left( 1 - \frac{ip}2 \right)}{\G\left(\frac{3-p-ip}4 \right) \G\left(\frac{3-p+ip}4 \right) \G\left(\frac{3+p-ip}4 \right) \G\left(\frac{3+p+ip}4 \right)}.
\end{equation}
We will show in \eqref{Afinalcong} that this gives the right hand side target from Theorem \ref{A} modulo $p^5$.

In the meantime, as in Sections \ref{E2F2} and \ref{G2}, we consider the series obtained from $A(p)$ by changing $p$ to an indeterminate $x$ and truncating at $\frac{p-1}2$,
\[
A(x) = \pFq{6}{5}{\frac54, \frac12, \frac{1-ix}2, \frac{1+ix}2, \frac{1-x}2, \frac{1+x}2}{,\frac14, 1+\frac{ix}2, 1-\frac{ix}2, 1+\frac{x}2, 1-\frac{x}2}{-1}_{\frac{p-1}2}.
\]
By Lemma \ref{cancelation}, $A(x)\in\Z_p[[x^4]]$, and so we have $A(x)= C_0 + C_4x^4 + C_8x^8 + \cdots$, for $C_i\in\Z_p$ where $C_0$ is our left hand side target \eqref{ALHS}.  Thus if $p\mid C_4$, then letting $x=p$ gives the desired congruence $A(p) \equiv C_0 \pmod{p^5}$.  

Considering instead 
\[
A'(x) = \pFq{6}{5}{\frac{5-p}4, \frac{1-p}2, \frac{1-ix}2, \frac{1+ix}2, \frac{1-x}2, \frac{1+x}2}{,\frac{1-p}4, 1+\frac{ix}2, 1-\frac{ix}2, 1+\frac{x}2, 1-\frac{x}2}{-1},
\]
we see that $A'(x)$ naturally truncates at $\frac{p-1}2$, so is actually a rational function in $x^4$ in $\Z_p[[x^4]]$.  Modulo $p$, $A(x)$ and $A'(z)$ have the same coefficients in $\Z_p[[x^4]]$.  However, by \eqref{6F5identity}, we see that 
\[
A'(x) = \frac{\G(1 + \frac{x}2)\G(1 - \frac{x}2)}{\G(\frac{3-p}2)\G(\frac{1-p}2)} \cdot \pFq{3}{2}{\frac{1-p}2, \frac{1-x}2, \frac{1+x}2}{, 1 + \frac{ix}2, 1 - \frac{ix}2}{1},
\]
where the ${}_3F_2$ series naturally truncates and is a rational function in $x^2$ and in $\Z_p[[x^2]]$.  The $\G$-factor in front however gives that $A'(x)=0$, since $\frac{3-p}2, \frac{1-p}2$ are negative integers.  Thus since modulo $p$, $A(x)$ and $A'(z)$ have the same coefficients in $\Z_p[[x^4]]$ we must have that $p \mid C_4$ (in fact all of the $C_{i}$), and so
\[
A(p) \equiv \pFq{6}{5}{\frac54, \frac12, \frac12, \frac12, \frac12, \frac12}{,\frac14, 1, 1, 1, 1}{-1}_{\frac{p-1}2} \pmod{p^5}.
\]

Fix $k=\frac{p-1}2$.  We now show that $A(p)$ is congruent to the right hand target modulo $p^5$.   Starting from \eqref{AGammas}, we first observe that \eqref{Gquotient} together with \eqref{Greflection} yields that
\[
\G\left(1 + \frac{p}2\right)\G\left(1 + \frac{p}2\right) = \frac{\pi}2 \cdot \frac{(\frac32)_k}{(1 - \frac{p}2)_k}.
\]
Furthermore, by \eqref{Gduplication} we have that
\[
\G\left(1 + \frac{ip}2\right)\G\left(1 - \frac{ip}2\right) = \frac{1}{\pi} \G\left( \frac12 + \frac{ip}4\right)\G\left( \frac12 - \frac{ip}4\right)\G\left( 1 + \frac{ip}4\right)\G\left( 1 - \frac{ip}4\right).
\]
Using \eqref{Gquotient}, we can thus rewrite $A(p)$ from \eqref{AGammas} as
\[
A(p) = \frac{(\frac32)_k (\frac{3-p+ip}{4})_k (\frac{3-p-ip}{4})_k}{(1 - \frac{p}2)_k (1 + \frac{ip}4)_k (1 - \frac{ip}4)_k}.
\]
We now use Lemma \ref{Gpquotient} for $\alpha\in\{\frac32, 1 -\frac{p}2, 1 + \frac{ip}{4}, 1 - \frac{ip}{4}, \frac{3-p+ip}{4}, \frac{3-p-ip}{4} \}$ to analyze the factors in terms of $\G_p$.  Note that since $p\equiv 1\pmod{4}$ we have $i\in\Z_p$.  When $\alpha = \frac32$, the factors $\alpha, \alpha+1,\ldots, \alpha + k - 2$ are not in $p\Z_p$, but $\alpha+k-1 = \frac{p}2$.  When  $\alpha = 1 -\frac{p}2$, none of the factors $\alpha+j$ are in $p\Z_p$.  Thus with \eqref{Gpfunctional} and \eqref{Gpreflection} we have,
\begin{equation}\label{firstA}
\frac{(\frac32)_k}{(1 - \frac{p}2)_k} = \frac{p}2 \cdot \frac{ \G_p(1 + \frac{p}2)\G_p(1 - \frac{p}2)}{\G_p\left( \frac12\right)\G_p\left( \frac32\right)} = p\cdot \G_p\left(1 + \frac{p}2\right)\G_p\left(1 - \frac{p}2\right).
\end{equation}

When $\alpha=1 \pm \frac{ip}{4}$, we have that none of the factors $\alpha+j$ are in $p\Z_p$, and so
\begin{equation}\label{secondA}
\frac{1}{(1 + \frac{ip}4)_k (1 - \frac{ip}4)_k}  = \frac{\G_p(1 + \frac{ip}4) \G_p(1 - \frac{ip}4)}{\G_p(\frac34 + \frac{(1+i)p}4)\G_p(\frac34 + \frac{(1-i)p}4)}.
\end{equation}

Similarly when $\alpha=\frac{3-p\pm ip}{4}$ none of the factors $\alpha+j$ are in $p\Z_p$, so
\begin{equation}\label{thirdA}
\left(\frac{3-p+ip}{4}\right)_k \left(\frac{3-p-ip}{4}\right)_k  = \frac{\G_p(\frac12 + \frac{ip}4)\G_p(\frac12 - \frac{ip}4)}{\G_p(\frac34 + \frac{(-1+i)p}4)\G_p(\frac34 + \frac{(-1-i)p}4)}.
\end{equation}

Together \eqref{firstA}, \eqref{secondA}, and \eqref{thirdA} give that
\[
A(p) = \frac{p \G_p(1 + \frac{p}2)\G_p(1 - \frac{p}2) \G_p(1 + \frac{ip}4) \G_p(1 - \frac{ip}4) \G_p(\frac12 + \frac{ip}4)\G_p(\frac12 - \frac{ip}4)}{\G_p(\frac34 + \frac{(1+i)p}4)\G_p(\frac34 + \frac{(1-i)p}4) \G_p(\frac34 + \frac{(-1+i)p}4) \G_p(\frac34 + \frac{(-1-i)p}4)}.
\]
Using the discussion in Remark \ref{LRext} we can analyze this quotient modulo $p^4$.  First, note that 
\[
\G_p\left(1 + \frac{p}2\right)\G_p\left(1 - \frac{p}2\right) \equiv \G_p(1)^2\left[1 + \frac14G_2(1)p^2 -\frac14 G_1(1)^2p^2 \right] \equiv 1 \pmod{p^4},
\]
using that $G_1(1)^2=G_2(1)$ by \eqref{Gfacts}.  Similarly,
\[ 
\G_p\left(1 + \frac{ip}4\right) \G_p\left(1 - \frac{ip}4\right) \equiv\G_p(1)^2\left[1 -\frac1{16}G_2(1)p^2 +\frac1{16} G_1(1)^2p^2 \right] \equiv 1 \pmod{p^4}.
\]
Also, since $\G_p(1/2)^2 = -1$ by \eqref{Gpreflection} we have
\[ 
\G_p\left(\frac12 + \frac{ip}4\right) \G_p\left(\frac12 - \frac{ip}4\right) \equiv\G_p\left(\frac12\right)^2\left[1 -\frac1{16}G_2\left(\frac12\right)p^2 +\frac1{16} G_1\left(\frac12\right)^2p^2 \right] \equiv -1 \pmod{p^4}.
\]
Using the same technique on the denominator we see that
\begin{align*}
\G_p\left(\frac34 + \frac{(i+1)p}4\right) \G_p\left(\frac34 - \frac{(i+1)p}4\right) & \equiv \G_p\left(\frac34\right)^2 \left[ 1 - \frac{i}{8} \left(G_1\left(\frac34\right)^2 - G_2\left( \frac34\right) \right) p^2\right] \pmod{p^4} \\
\G_p\left(\frac34 + \frac{(i-1)p}4\right) \G_p\left(\frac34 - \frac{(i-1)p}4\right) & \equiv \G_p\left(\frac34\right)^2 \left[ 1 + \frac{i}{8} \left(G_1\left(\frac34\right)^2 - G_2\left( \frac34\right) \right) p^2\right] \pmod{p^4} \\
\end{align*}
and so
\[
\G_p\left(\frac34 + \frac{(1+i)p}4\right)\G_p\left(\frac34 + \frac{(1-i)p}4\right) \G_p\left(\frac34 + \frac{(-1+i)p}4\right) \G_p\left(\frac34 + \frac{(-1-i)p}4\right) \equiv \G_p\left(\frac34\right)^4 \pmod{p^4}.
\]
Putting this together, the factor of $p$ in front gives the desired congruence modulo $p^5$
\begin{equation}\label{Afinalcong}
A(p) \equiv \frac{-p}{\G_p\left(\frac34 \right)^4} = -p\cdot \G_p\left( \frac14\right)^4 \pmod{p^5}.
\end{equation}
\end{proof}

%\begin{remark}
%We note that the original statement of van Hamme's (A.2) can be proven together using the following identities of Clausen and Gauss, respectively \cite{AAR}
%\begin{equation}\label{clausen}
%\pFq{3}{2}{2a, 2b, a+b}{, 2a+2b, a+b+\frac12}{z} = \pFq{2}{1}{a,b}{,a+b+\frac12}{z}^2,
%\end{equation}
%and
%\begin{equation}\label{gauss}
%\pFq{2}{1}{a,b}{,c}{1} = \frac{\G(c)\G(c-a-b)}{\G(c-a)\G(c-b)}.
%\end{equation}
%\end{remark}

\section{Conjectures}\label{conjectures}

The following more general van Hamme type congruence conjectures are supported by computational evidence computed with Ling Long and Hao Chen using Sage.  Note that some of these conjectures extend van Hamme's conjectures in the $r=1$ case, which motivated several of the theorems in this paper.  

\begin{itemize}
\item[(A.3)]
\[
\left \{
\begin{array}{lll} 
S(\frac{p^r-1}2) \equiv -p\Gamma_p(\frac14)^4 S(\frac{p^{r-1}-1}2)&\pmod{p^{5r}} & p\equiv 1\pmod4, \; r\ge 1\\
\\
S(\frac{p-1}2) \equiv 0 &\pmod{p^{3}} &p\equiv 3\pmod4\\
\\
S(\frac{p^r-1}2) \equiv p^4 S(\frac{p^{r-2}-1}2)&\pmod{p^{5r-2}} &p\equiv 3\pmod4, \; r\ge 2\\
\end{array} \right.
\]

\item[(B.3)]
\[
\left \{\begin{array}{lll} 
S(\frac{p^r-1}2)=-p\Gamma_p(\frac12)^2 S(\frac{p^{r-1}-1}2)&\mod p^{3r}&p\equiv 1\pmod4\\ 
\\
S(\frac{p-1}2)=-p\Gamma_p(\frac12)^2  &\mod p^{3} & p\equiv 3\pmod 4\\ 
\\
S(\frac{p^r-1}2)=p^2 S(\frac{p^{r-2}-1}2)&\mod p^{3r-2}&p\equiv 3\pmod 4, \; r\ge 2 \\
\end{array} \right.
\]

\item[(C.3)] 
\[
S\left(\frac{p^r-1}2\right) \equiv p S\left(\frac{p^{r-1}-1}2\right) \pmod{p^{4r}}
\]

\item[(D.3)]
\[
\left \{\begin{array}{lll} 
S(\frac{p^r-1}3) \equiv -p{\Gamma_p(\frac13)^9} S(\frac{p^{r-1}-1}3)&\pmod{p^{6r}}&p\equiv 1\pmod 3\\ 
\\
S(\frac{p^2-1}3) \equiv 0  &\pmod{p^{4}} &p\equiv 2\pmod 3\\ 
\\
S(\frac{p^r-1}3) \equiv p^4 S(\frac{p^{r-2}-1}3)&\pmod{p^{2r+1}} &p\equiv 2\pmod 3, \; r\ge 4 \text{ even}\\ 
\\
S(\frac{p^r-2}3)=-p^5 S(\frac{p^{r-2}-2}3)&\pmod{p^{2r}}&p\equiv 2\pmod 3, \; r\ge 3 \text{ odd}\\
\end{array} \right.
\]

\item[(E.3)]
\[
\left \{\begin{array}{lll} 
S(\frac{p^r-1}3) \equiv pS(\frac{p^{r-1}-1}3)&\pmod{p^{3r}}&p\equiv 1\pmod 3\\ 
\\
S(\frac{p^r-1}3) \equiv p^2 S(\frac{p^{r-2}-1}3)&\pmod{p^{3r-2}}&p\equiv 2\pmod 3, \; r\ge 2 \text{ even}\\ 
\\
S(\frac{p-2}3)\equiv 0 &\pmod{p} & p\equiv 2\pmod 3 \\
\\
S(\frac{p^r-2}3)\equiv p^2 S(\frac{p^{r-2}-2}3)&\pmod{p^{3r-1}} &p\equiv 2\pmod 3, \;r \ge  3 \text{ odd}\\
\end{array} \right.
\]

\item[(F.3)]
\[
\left \{\begin{array}{lll} 
S(\frac{p^r-1}4) \equiv (-1)^{\frac{p^2-1}8}pS(\frac{p^{r-1}-1}4)&\pmod{p^{3r}}&p\equiv 1\pmod 4\\ 
\\
S(\frac{p^r-1}4) \equiv p^2 S(\frac{p^{r-2}-1}4)&\pmod{p^{3r-2}} &p\equiv 3\pmod 4,\; r\ge 2 \text{ even}\\ 
\\
S(\frac{p^r-3}4)=p^2 S(\frac{p^{r-2}-3}4)&\mod p^{r}&p\equiv 3\pmod 4, \; r\ge  3 \text{ odd}\\
\end{array} \right.
\]

\item[(G.3)]
\[
\left \{\begin{array}{lll} 
S(\frac{p^r-1}4) \equiv -(-1)^{\frac{p^2-1}8}p\Gamma(\frac12)\Gamma_p(\frac14)^2S(\frac{p^{r-1}-1}4)&\pmod{p^{4r}} &p\equiv 1\pmod 4\\ 
\\
S(\frac{p^r-1}4)\equiv -p^3 S(\frac{p^{r-2}-1}4)&\pmod{p^{4r-2}} &p\equiv 3\pmod 4,\; r\ge 2 \text{ even}\\
\\
S(\frac{p^r-3}4) \equiv -p^3 S(\frac{p^{r-2}-3}4)&\pmod{p^{r+1}}&p\equiv 3\pmod4, \; r\ge 3 \text{ odd}\\
\end{array} \right.
\]

\item[(H.3)]
\[
\left \{\begin{array}{lll} 
S(\frac{p^r-1}2) \equiv -{\Gamma_p(\frac14)^4} S(\frac{p^{r-1}-1}2)&\pmod{p^{3r}} &p\equiv 1\pmod 4\\ 
\\
S(\frac{p-1}2)\equiv 0 &\pmod{p^2} & p\equiv 3\pmod 4\\
\\
S(\frac{p^r-1}2)\equiv p^2 S(\frac{p^{r-2}-1}2)&\pmod{p^{3r-1}} &p\equiv 3\pmod4,\; r\ge 2
\end{array} \right.
\]

\item[(I.3)] 
\[
S\left(\frac{p^r-1}2\right)=2p^{2r} \pmod{p^{2r+1}}
\]

\item[(J.3)] 
\[
S\left(\frac{p^r-1}2\right) \equiv (-1)^{\frac{p-1}2}pS\left(\frac{p^{r-1}-1}2\right) \pmod{p^{4r}}
\]

\item[(K.3)] 
For $p>5$,
\[
\left \{\begin{array}{lll} 
S\left(\frac{p-1}2\right) \equiv -5(-1)^{\frac{p-1}2}p &\pmod{p^4}&\\
\\
S\left(\frac{p^r-1}2\right) \equiv -(-1)^{\frac{p-1}2}pS\left(\frac{p^{r-1}-1}2\right) &\pmod{p^{4r}}&
\end{array} \right.
\] 

\item[(L.3)] 
\[
S\left(\frac{p^r-1}2\right) \equiv (-1)^{\frac{p-1}2}(-1)^{\frac{p^2-1}8}pS\left(\frac{p^{r-1}-1}2\right) \equiv \left(\frac{-2}p\right)pS\left(\frac{p^{r-1}-1}2\right) \pmod{p^{3r}}. 
\] 

\end{itemize}

%%%%%%%
\section{Acknowledgements}The author would like to thank Ling Long for numerous helpful conversations which provided the ideas behind this work, and Robert Osburn for many inspiring conversations and for introducing her to the van Hamme conjectures.  The author also thanks Tulane University for hosting her while working on this project, and Sage Days 56, where she ran computations related to this project with Ling Long and Hao Chen. 

\bibliography{vHbib}

\begin{thebibliography}{10}

\bibitem{AhlgrenOno}
Scott Ahlgren and Ken Ono.
\newblock Modularity of a certain calabi-yau threefold.
\newblock {\em Monatshefte f{\"u}r Mathematik}, 129(3):177--190, 2000.

\bibitem{AAR}
George~E. Andrews, Richard Askey, and Ranjan Roy.
\newblock {\em Special functions}, volume~71 of {\em Encyclopedia of
  Mathematics and its Applications}.
\newblock Cambridge University Press, Cambridge, 1999.

\bibitem{BorweinBorwein}
Jonathan~M. Borwein and Peter~B. Borwein.
\newblock {\em Pi and the {AGM}}.
\newblock Canadian Mathematical Society Series of Monographs and Advanced
  Texts. John Wiley \& Sons Inc., New York, 1987.
\newblock A study in analytic number theory and computational complexity, A
  Wiley-Interscience Publication.

\bibitem{CDLNS}
Sarah Chisholm, Alyson Deines, Ling Long, Gabriele Nebe, and Holly Swisher.
\newblock $p$--adic analogues of {R}amanujan type formulas for 1/$\pi$.
\newblock {\em Mathematics}, 1(1):9--30, 2013.

\bibitem{ChudnovskyChudnovsky}
David~V. Chudnovsky and Gregory~V. Chudnovsky.
\newblock Approximations and complex multiplication according to {R}amanujan.
\newblock In {\em Ramanujan revisited ({U}rbana-{C}hampaign, {I}ll., 1987)},
  pages 375--472. Academic Press, Boston, MA, 1988.

\bibitem{Dwork}
Bernard Dwork.
\newblock A note on the {$p$}-adic gamma function.
\newblock In {\em Study group on ultrametric analysis, 9th year: 1981/82, {N}o.
  3 ({M}arseille, 1982)}, pages Exp. No. J5, 10. Inst. Henri Poincar\'e, Paris,
  1983.

\bibitem{Karlsson}
Per~W. Karlsson.
\newblock Clausen's hypergeometric function with variable {$-1/8$} or {$-8$}.
\newblock {\em Math. Sci. Res. Hot-Line}, 4(7):25--33, 2000.

\bibitem{Kilbourn}
Timothy Kilbourn.
\newblock An extension of the {A}p\'ery number supercongruence.
\newblock {\em Acta Arith.}, 123(4):335--348, 2006.

\bibitem{Long}
Ling Long.
\newblock Hypergeometric evaluation identities and supercongruences.
\newblock {\em Pacific J. Math.}, 249(2):405--418, 2011.

\bibitem{LongRamakrishna}
Ling Long and Ravi Ramakrishna.
\newblock Some supercongruences occurring in truncated hypergeometric series.
\newblock {\em preprint}.
\newblock arXiv:1403.5232.

\bibitem{McCarthyOsburn}
Dermot McCarthy and Robert Osburn.
\newblock A {$p$}-adic analogue of a formula of {R}amanujan.
\newblock {\em Arch. Math. (Basel)}, 91(6):492--504, 2008.

\bibitem{Morita}
Yasuo Morita.
\newblock A p-adic analogue of the-function.
\newblock {\em J. Fac. Sci. Univ. Tokyo}, 22:255--266, 1975.

\bibitem{Mortenson}
Eric Mortenson.
\newblock A {$p$}-adic supercongruence conjecture of van {H}amme.
\newblock {\em Proc. Amer. Math. Soc.}, 136(12):4321--4328, 2008.

\bibitem{vanHamme}
Lucien van Hamme.
\newblock Some conjectures concerning partial sums of generalized
  hypergeometric series.
\newblock {\em Lecture Notes in Pure and Appl. Math}, 192:223--236, 1997.

\bibitem{vanGeemenNygaard}
Bert Vangeemen and Niels~O Nygaard.
\newblock On the geometry and arithmetic of some siegel modular threefolds.
\newblock {\em Journal of Number Theory}, 53(1):45--87, 1995.

\bibitem{Verrill}
H.~A. Verrill.
\newblock Arithmetic of a certain {C}alabi-{Y}au threefold.
\newblock In {\em Number theory ({O}ttawa, {ON}, 1996)}, volume~19 of {\em CRM
  Proc. Lecture Notes}, pages 333--340. Amer. Math. Soc., Providence, RI, 1999.

\bibitem{Whipple}
F.~J.~W. Whipple.
\newblock On well-poised series, generalized hypergeometric series having
  parameters in pairs, each pair with the same sum.
\newblock {\em Proc. London Math. Soc.}, s2-24(1):247--263, 1926.

\bibitem{Zudilin}
Wadim Zudilin.
\newblock Ramanujan-type supercongruences.
\newblock {\em J. Number Theory}, 129(8):1848--1857, 2009.

\end{thebibliography}
\bibliographystyle{plain}

\end{document}